\documentclass[12pt,a4paper,twoside]{amsart}

\usepackage{amsmath}
\usepackage{amssymb}

\numberwithin{equation}{section}
\numberwithin{figure}{section}
\numberwithin{table}{section}

\newtheorem{theorem}{Theorem}[section]

\newtheorem{proposition}[theorem]{Proposition}
\newtheorem{lemma}[theorem]{Lemma}

\theoremstyle{definition} %
\newtheorem{remark}[theorem]{Remark}
\newtheorem{definition}[theorem]{Definition}

\newcommand\R{\mathbb R}

\newcommand\FD{{\mathcal A}}

\newcommand\odd{{C^\infty_{\rm odd}(S^2,\R)}}

\newcommand\good{{amply negative }}

\newcommand\Good{{Amply negative }}

\newcommand\Gv{{\Sigma}}

\newcommand\Guil{\Psi^{f}}

\newcommand\Se{X}

\newcommand\ppi{P}

\title{A Zoll counterexample to a geodesic length conjecture}

\author[F.~Balacheff]{Florent Balacheff$^1$} \address{Section de
Math\'ematiques, Universit\'e de Gen\`eve, Suisse}
\email{florent.balacheff@math.unige.ch} \thanks{$^1$Supported by the
Swiss National Science Foundation}

\author[C.~Croke]{Christopher Croke$^2$} \address{ Department of
Mathematics, University of Pennsylvania, Philadelphia, PA 19104-6395
USA} \email{ccroke@math.upenn.edu} \thanks{$^2$Supported by NSF grants
DMS 02-02536 and DMS 07-04145}

\author[M.~Katz]{Mikhail~G.~Katz$^{3}$} %
\thanks{$^{3}$Supported by the Israel Science Foundation (grants 84/03
and 1294/06)} \address{Mikhail G. Katz, Department of Mathematics, Bar
Ilan University, Ramat Gan 52900 Israel}
\email{katzmik ``at'' math.biu.ac.il}

\begin{document}



\date{\today}

\begin{abstract}
We construct a counterexample to a conjectured inequality~$L\leq 2D$,
relating the diameter~$D$ and the least length~$L$ of a nontrivial
closed geodesic, for a Riemannian metric on the~$2$-sphere.  The
construction relies on Guillemin's theorem concerning the existence of
Zoll surfaces integrating an arbitrary infinitesimal odd deformation
of the round metric.  Thus the round metric is not optimal for the
ratio~$L/D$.
\end{abstract}

\maketitle
\tableofcontents

\section{Zoll surfaces and Guillemin deformation}

Given a Riemannian metric on the~$2$-sphere, we consider its
diameter~$D$ and the length~$L$ of its shortest nontrivial closed
geodesic.  The first inequality relating the two invariants was
obtained by the second-mentioned author \cite{Cr}, who proved the
bound~$L \leq 9 D$.  The constant in the inequality was successively
improved by M.~Maeda \cite{Ma94}, A.~Nabutovsky and R.~Rotman
\cite{NR02}, and S. Sabourau \cite{Sa0}.  The best known bound is~$L
\leq 4 D$.  Nabutovsky and Rotman conjectured the inequality~$L \leq 2
D$ \cite[Introduction]{NR02}, meaning that the round metric of~$S^2$
is optimal for the relationship between these two invariants.  We give
a few examples of surfaces satisfying the case of equality~$L=2D$:

\begin{enumerate}
\item
a surface of revolution in~$\R^3$ obtained from an ellipse with major
axis on the~$x$-axis;
\item
a circular ``pillow'', obtained by doubling the flat unit disk;
\item
a more general pillow obtained by doubling the region enclosed by a
closed curve of constant width in the plane;
\item
rotationally invariant Zoll surfaces.
\end{enumerate}

The existence of such diverse examples may have led one to expect that
none of these metrics are optimal for the ratio~$L/D$.

It turns out that a counterexample to the inequality~$L\leq 2D$ may be
found among Zoll surfaces, namely surfaces all of whose geodesics are
closed, and whose prime geodesics all have equal length~$2\pi$.  More
precisely, while the rotationally symmetric Zoll surfaces do satisfy
(the boundary case of equality of) the conjectured inequality, there
exist other families of Zoll surfaces such that~$L>2D$.  Such surfaces
can be obtained as smooth variations of the round metric.

Let~$(S^2,g_0)$ be the~$2$-sphere endowed with the round metric.
Denote by~$a: S^2 \to S^2$ its antipodal map.  Let~$\odd$ be the space
of smooth odd functions on~$S^2$, i.e.~smooth real valued
functions~$f$ satisfying~$f \circ a = -f$.  The following existence
theorem for Zoll surfaces is due to V. Guillemin \cite{Gu76}.

\begin{theorem}[Guillemin]
\label{Gu}
For every~$f \in \odd$, there exists a smooth one-parameter
family~$g_t=\Guil_t g_0$ of smooth Zoll metrics such that~$\Guil_0=1$,
the conformal factor~$\Guil_t$ satisfies~$(d \Guil_t /dt) |_{t=0}
^{\phantom{I}} =f$, and all prime periodic geodesics of $(S^2, g_t)$
have length $2\pi$.
\end{theorem}

Note that this result is a converse to P.~Funk's theorem \cite{Fu13},
to the effect that a smooth variation~$g_t=\Phi_t g_0$ of the round
metric by smooth Zoll metrics necessarily satisfies~$(d \Phi_t/dt)
|_{t=0}^{\phantom{I}} \in \odd$.  A survey of Zoll surfaces appeared
in \cite[Chapter~4]{Bes}, see also \cite{LM}.

We exploit such Guillemin deformations to show that the round metric
is not even a local maxinum of the ratio~$L/D$ among Zoll surfaces.
The precise statement of our result relies on the notion of a Y-like
set.

\begin{definition}
\label{ylike}
A subset of the unit circle is called {\em Y-like\/} if it contains a
triple of vectors~$\{u,v,w\}$ such that there exist positive real
numbers~$a>0, b>0, c>0$ satisfying~$au+bv+cw=0$.  A subset of the unit
tangent bundle~$US^2$ of~$S^2$ will be called {\em Y-like} if its
intersection with the unit tangent vectors at~$p$ is {\em Y-like} for
every~$p\in S^2$.
\end{definition}

Note that a subset of the unit circle is~$Y$-like if and only if every
open semicircle contains an element of the set.

We will denote by~$ds_0$ the element of length for the round
metric~$g_0$ on the sphere.  The notion of an \good function is
motivated in Remark~\ref{26} below.

\begin{definition}
An odd function~$f$ is called {\em \good\/} if the set of unit tangent
directions to great half-circles~$\tau$ satisfying~$\int_{\tau} f
ds_0<0$, is a Y-like subset of~$US^2$.
\end{definition}

\begin{theorem}
\label{14}
If $f$ is an \good function then the smooth variation~$\{g_t\} =
\Guil_t g_0$ of the round metric~$g_0$ by smooth Zoll metrics
satisfies $L(g_t) >2 D(g_t)$ for sufficiently small~$t>0$.
\end{theorem}

Combined with the existence of \good functions proved in
Section~\ref{five}, our theorem yields the desired counterexample.

These metrics also provide a counterexample to another conjecture of
Nabutovsky and Rotman \cite[Conjecture~1, p.~13]{NR05}.  Their
conjecture would imply that for every point~$p$ of a closed Riemannian
manifold~$(M,g)$, there is a nontrivial geodesic loop at~$p$ of length
at most~$2 D(g)$.  Here a geodesic loop is a geodesic segment with
identical endpoints.  This conjecture is easily seen to be true for
non-simply-connected manifolds, by exploiting non-contractible loops,
cf.~\cite{SGT}.  In our examples, the shortest geodesic loop at
every~$p$ has length~$2\pi$, while the diameter is strictly smaller
than~$\pi$.

Sections~\ref{two} and \ref{three} contain a proof of Theorem~\ref{14}
modulo on the existence of \good functions.  The existence of the
latter is verified in Sections~\ref{four} and \ref{five}.

\section{\Good odd functions}
\label{two}

Our goal is to find \good functions~$f\in \odd$, such that the
corresponding Guillemin deformation~$g_t$ of the standard round
metric~$g_0$ satisfies~$D(g_t)<\pi$ for~$t$ small enough (while all
geodesics remain closed of length~$2\pi$).  By the compactness of the
unit tangent circle bundle~$US^2$, we obtain the following lemma.

\begin{lemma}
For every \good function~$f$ there is a constant~$\nu(f)>0$ with the
following property.  For every~$(p,v)\in US^2$, there is a great
half-circle~$\tau$ issuing from~$p\in S^2$, forming an acute angle
with~$v$, and satisfying~$\int_{\tau} f ds_0<-\nu(f)$.
\end{lemma}

Denote by~$L_0$ the length functional with respect to the round
metric~$g_0$.  Given a geodesic segment~$\gamma$ of
length~$L_0(\gamma)<\pi$, we denote by~$\ppi_{\gamma}$
the~$1$-parameter family of piecewise geodesic paths with the
following two properties:

\begin{itemize}
\item
the path joins the endpoints of~$\gamma$,
\item
the path consists of a pair of imbedded geodesic segments of equal
length.
\end{itemize}

Elements of~$\ppi_{\gamma}$ are parametrized by the non-smooth
midpoint of the piecewise geodesic path, which traces out the
equidistant great circle of the two endpoints.  We let~$S\ppi_{\gamma}
\subset \ppi_{\gamma}$ be the closed subfamily consisting of the
shorter paths, namely
\[
S\ppi_{\gamma} = \left\{ \left. \tau\in \ppi_{\gamma} \right|
L_0(\tau)\leq \pi \right\}.
\]
If~$\gamma$ is a great semi-circle, define~$\ppi_{\gamma}$ to be the
circular family of great half-circles joining the endpoints
of~$\gamma$, and the subfamily~$S\ppi_{\gamma}$ to be the family of
paths forming either an acute or a right angle with~$\gamma$ at the
endpoints.  The following lemma is obvious but crucial.

\begin{lemma}
\label{crucial}
The family~$S\ppi_{\gamma}$ for a geodesic segment~$\gamma$
with~$L_0(\gamma)=\pi$ is the limit of the families~$S\ppi_{\gamma_i}$
for subarcs~$\gamma_i$ of~$\gamma$ of length tending to~$\pi$.  In
fact, if~$\gamma_i$ is any sequence of minimizing geodesic segments
converging to~$\gamma$, then~$S\ppi_{\gamma_i}$ converges
to~$S\ppi_\gamma$.
\end{lemma}

Our main technical tool in the next section will be the following
result.

\begin{lemma}
\label{24}
If~$f$ is \good then there is an~$\epsilon>0$ so that for all geodesic
segments~$\gamma$ with~$\pi-\epsilon \leq L_0(\gamma) \leq \pi$, there
is a path~$\tau\in S\ppi_{\gamma}$ with~$\int_{\tau} f ds_0<-\nu(f)$.
\end{lemma}

\begin{proof}
If no such~$\epsilon$ exists, then there is a sequence~$\{ \gamma_i
\}$ with~$L_0(\gamma_i)<\pi$ and~$L_0(\gamma_i) \to \pi$ such that all
$\tau\in S\ppi_{\gamma_i}$ satisfy~$\int_{\tau} f ds_0\geq -\nu(f)$.
There is a convergent subsequence such that~$\gamma_i'(0) \to
\gamma'(0)$ with~$L_0(\gamma)=\pi$.  By Lemma~\ref{crucial}, the
family~$S\ppi_{\gamma_i}$ converges to~$S\ppi_{\gamma}$.  By the
continuity of~$f$, for every~$\tau\in S\ppi_{\gamma}$ we have
$\int_{\tau} f ds_0\geq -\nu(f)$, contradicting the assumption that
the function~$f$ is \good\!.
\end{proof}

\begin{remark}
\label{26}
Given a piecewise geodesic~$\tau$ over which the integral of~$f$ is
negative, we will show in the next section that the length of~$\tau$
decreases under the Guillemin deformantion.  If, in addition, the
curve~$\tau$ has length at most~$\pi$ with respect to the metric
$g_0$, then the length with respect to the metric~$g_t$ will be
shorter than~$\pi$.  That is why we need to work with piecewise
geodesics specifically in~$S\ppi_{\gamma}$.  In order to make the
continuity argument above work, one needs to find in each
$S\ppi_{\gamma}$, a curve~$\tau$ over which~$f$ integrates negatively.
This leads to the \good condition we introduced.
\end{remark}

\section{Diameter of Guillemin deformation}
\label{three}

Let~$\Guil$ be the conformal factor of the Guillemin deformation, as
in Theorem~\ref{Gu} above.  Thus, the metric~$g_t=\Guil_t g_0$ is
Zoll, while~$\Guil_0=1$ and~$(d \Guil_t /dt) |_{t=0} ^{\phantom{I}}
=f$.  Consider the arclength parametrisation~$\tau(s)$ of a
path~$\tau\subset S^2$ for the round metric~$g_0$.

\begin{lemma}
\label{31}
The energy~$E_t(\tau)$ of a path~$\tau \subset S^2$ for the metric
$g_t$ satisfies
\[
\left. \frac{d E_t}{dt} \right|_{t=0}^{\phantom{I}} = \int_\tau f
\circ \tau \; ds_0.
\]
\end{lemma}

\begin{proof}
We have
\[
\begin{aligned}
\frac{d}{dt}E_t(\tau) &=\frac d {dt}
\int_0^{L_0(\tau)}g_t(\tau'(s),\tau'(s)) ds \\&= \frac{d}{dt}
\int_\tau \Guil_t \circ \tau \; ds_0 \\&= \int_\tau \left( \tfrac d
{dt}\Guil_t \right) \circ \tau \; ds_0 \\&= \int_\tau f \circ \tau \;
ds_0
\end{aligned}
\]
at~$t=0$.
\end{proof}

\begin{proposition}
\label{23}
If~$f$ is \good\!, then the associated Guillemin deformation~$g_t =
\Guil_t g_0$ as in Theorem~\ref{Gu} satisfies~$D(g_t)<\pi$ for all
sufficiently small~$t>0$.
\end{proposition}

\begin{proof}
Denote by~$L_t$ and~$d_t$ the length and the distance with respect to
the metric~$g_t$.  Let~$\epsilon>0$ be chosen as in Lemma~\ref{24},
and let~$\FD_{\epsilon}\subset S^2\times S^2$ be the set of nearly
antipodal pairs, defined by setting
\[
\FD_\epsilon = \left\{ (p,q)\in S^2\times S^2 \left|
d_0^{\phantom{I}}(p,q) \geq \pi-\epsilon \right. \right\}
\]
By continuity, there is a~$\delta>0$ such that whenever~$0<t< \delta$,
we have
\begin{equation}
\label{31b}
d_t(p,q)<\pi \hbox{ for all } (p,q)\not\in \FD_\epsilon.
\end{equation}
Now let~$(p,q)\in \FD_{\epsilon}$, and~$\gamma$ a minimizing geodesic
joining them.  Let
\[
N(\gamma)=\left\{ \tau\in S\ppi_{\gamma}\left| \int_{\tau} f
ds_0\leq-\nu(f) \right. \right\},
\]
and let~$N=\{\tau\in N(\gamma)| \pi-\epsilon \leq L_0(\gamma)\leq
\pi\}$.  By Lemma \ref{24}, whenever
\[
\pi-\epsilon\leq L_0(\gamma)\leq \pi,
\]
the set~$N(\gamma)$ is non-empty.  Furthermore, the sets~$N$
and~$N(\gamma)$ are compact.  Now for small~$t>0$, define a continuous
function~$F:N\times \R\to \R$ by setting~$F(\tau,t) =
\tfrac{dE_t(\tau)}{dt}$.  By Lemma~\ref{31} and the definition of~$N$,
we have~$F(\tau,0)\leq -\nu(f)$.  Hence by the compactness of~$N$ and
the continuity of~$F$ there is a real~$\delta'>0$ so that for
all~$0\leq t \leq \delta'$ and all~$\tau \in N$, we have~$F(\tau,t)<
-\frac 1 2 \nu(f)$.  Therefore the energy given by the expression
\[
\int_0^{L_0(\tau)}g_t(\tau'(s),\tau'(s)) ds_0
\]
is strictly decreasing in~$t$.  Hence for~$0<t\leq \delta'$, it is
strictly smaller than the quantity
\[
\int_0^{L_0(\tau)}g_0(\tau'(s),\tau'(s)) ds_0=L_0(\tau).
\]
In particular, we obtain for~$0<t\leq \delta'$,
\begin{eqnarray}
\nonumber L_t(\tau)&=&\int_0^{L_0(\tau)}\sqrt{g_t(\tau'(s),\tau'(s))}
ds_0\\ \nonumber &\leq & L_0(\tau)^{\frac 1
2}\Big(\int_0^{L_0(\tau)}g_t(\tau'(s),\tau'(s)) ds_0\Big)^{\frac 1
2}\\ \nonumber &<& L_0(\tau). \nonumber \end{eqnarray}
Thus for each pair~$(p,q) \in \FD_\epsilon$ and every~$0<t\leq
\delta'$, there is a path~$\tau$ from~$p$ to~$q$
with~$L_t(\tau)<L_0(\tau)\leq \pi$.  Hence~${\rm dist}_t(p,q)<\pi$.
Combined with \eqref{31b}, this yields the diameter bound~$D(g_t)<\pi$
whenever~$0<t<\min\{\delta,\delta'\}$, proving the proposition as well
as Theorem~\ref{14}.
\end{proof}

\section{Fine sets and their properties}
\label{four}

Recall that an open hemisphere is an open ball of radius~$\pi/2$
centered at any point of the unit sphere.  The construction of \good
functions in Section~\ref{five} exploits {\em fine\/} sets, in the
following sense.

\begin{definition}
A spherical pointset~$\Se$ is called {\em fine\/} if the following
three conditions are satisfied:
\begin{enumerate}
\item
\label{coll2}
no triple of~$\Se$ is collinear;
\item
\label{concur2}
no triple of great circles~$pp'$, where~$p,p' \in \Se$, is concurrent
other than at points of~$\Se$ (as well as their antipodal points);
\item
\label{ample}
every open hemisphere contains at least~$3$ of the points of~$\Se$.
\end{enumerate}
\end{definition}

Note that the non-collinearity implies, in particular, that~$\Se$
contains no pair of antipodal points.  Meanwhile, condition
\eqref{ample} implies that at every point of the sphere, there is a
Y-like set of tangent directions leading to points of~$\Se$.

To see that fine sets exist, start with the set of~$4$ vertices of the
regular inscribed tetrahedron.  This gives a set with at least one
point in every open hemisphere.  We replace each point of the
tetrahedron by a generic triple of nearby points.  The
non-collinearity and non-concurrency follow from genericity, and
property \eqref{ample} follows by construction.

\begin{definition}
Given a fine pointset~$\Se$, choose~$\epsilon(\Se)>0$ such that:
\begin{enumerate}
\item
the closed~$\epsilon(\Se)$ balls centered at the points of~$\Se \cup
-\Se$ are disjoint;
\item
there are at least 3 points of~$\Se$ in~$B(p,\pi/2-\epsilon(\Se))$ for
every~$p\in S^2$.
\end{enumerate}
\end{definition}

We note that property \eqref{ample} of fineness along with standard
compactness arguments shows that such a positive~$\epsilon(\Se)$
exists.

\begin{lemma}
\label{thereistau}
Let~$\Se$ be a fine set and choose~$\epsilon(\Se)$ as above.
Let~$\Gv$ the set of unit vectors in~$US^2$ tangent to geodesic
segments~$\tau$ of length~$\pi$ satisfying the following two
conditions:
\begin{itemize}
\item
$\tau(0,\pi)\cap -\Se=\emptyset$;
\item
$\tau (\epsilon(\Se),\pi- \epsilon(\Se))\cap \Se \not= \emptyset$.
\end{itemize}
Then~$\Gv$ is {\em Y-like}.
\end{lemma}

\begin{proof}
Fix a unit vector~$v$ at~$p\in S^2$ and let~$\gamma$ be the
corresponding geodesic segment of length~$\pi$.  We need to find
a~$w\in \Gv$ at~$p$ making an acute angle with~$v$.  Let~$H$ be the
(closed) hemisphere obtained as the union of the~$\tau\in
S\ppi_{\gamma}$.  Then by assumption there are at least three points
of~$\Se$ (call them~$p_{1}$,~$p_{2}$, and~$p_{3}$) in the interior
of~$H$ and at a distance greater than~$\epsilon(\Se)$ from the
boundary of~$H$ hence the endpoints of~$\gamma$.  Hence there are at
least 3 geodesic segment ~$\tau_1$,~$\tau_2$, and~$\tau_3$ in the
interior of~$S\ppi_{\gamma}$ passing through the points~$p_{i}, i=1,
2, 3$.  If two of these paths coincide (say~$\tau_1=\tau_2$)
then~$\tau_1$ passes through~$p_{1}$ and~$p_{2}$ so the initial
point~$p$ of~$\gamma$ is not in~$\Se$ and~$\tau_1$ avoids~$-\Se$ (by
condition 1 of being fine). If all of these paths are pairwise
distinct and also pass through points of~$-\Se$
(say~$-p_{4}$,$-p_{5}$, and~$-p_{6}$ respectively) then the initial
point~$p$ of~$\gamma$ would lie on the 3 great circles~$p_1p_4$,
$p_2p_5$,~$p_3p_6$ which contradicts either condition 2 (if~$p\notin
\Se \cup -\Se$) or condition 1 (if~$p\in \Se\cup -\Se$).  Thus we see
that there is a~$\tau$ in the interior of~$S\ppi_{\gamma}$ containing
an element of~$\Se$ at least~$\epsilon(\Se)$ from the endpoints, and
no element of~$-\Se$ in its interior.  The tangent vector of~$\tau$ is
the~$w$ we seek.
\end{proof}

\section{Existence of \good functions}
\label{five}

The goal of this section is to prove the following proposition.

\begin{proposition}
\label{existgood}
There exist \good functions.
\end{proposition}

We will costruct such functions by defining odd functions that
approximate the sum of~$\pm\delta$ (Dirac delta) functions centered at
points of~$-\Se$ and~$\Se$ for a fine set~$\Se$.  For our
approximate~$\delta$ functions we take for each~$p\in S^2$ the smooth
function~$\delta^\epsilon_p$ with support included in the
ball~$B(p,\epsilon)$ with
\[
\delta_p^\epsilon(q)= \exp(1/\epsilon)\cdot\exp
\left(\frac{1}{d(p,q)-\epsilon}\right)
\]
for~$q \in B(p,\epsilon)$.

We will use the following (nearly obvious) lemma.

\begin{lemma}
if~$\gamma$ is a diameter of~$B(p,\epsilon)$ (i.e. geodesic through
the center of length~$2\epsilon$) and~$\tau$ is any geodesic segment
in~$B(p,\epsilon)$ then~$\int_\tau \delta^\epsilon_p\leq \int_\gamma
\delta^\epsilon_p$ with equality holding if and only if~$\tau$ is also
a diameter.
\end{lemma}

\begin{proof}
To see this (since~$\delta^\epsilon_p\geq 0$) we can
assume (by extending~$\tau$ if needed) that~$\tau$ runs from a
boundary point to a boundary point and has length~$2l<2\epsilon$.
Since for~$t\leq l$ we have~$d(p,\tau(t))\geq
\epsilon-t=d(p,\gamma(t))$ we have
\[
\int_\tau \delta^\epsilon_p=2\int_0^l \delta^\epsilon_p(\tau(t))dt\leq
2\int_0^l \delta^\epsilon_p(\gamma(t))<2\int_0^\epsilon
\delta^\epsilon_p(\gamma(t))=\int_\gamma \delta^\epsilon_p.
\]
\end{proof}

We are now ready to define our functions.

\begin{definition}
For~$\epsilon(\Se)>\epsilon >0$ set
\[
f^\epsilon_\Se =\sum_{p_i\in\Se} (\delta^\epsilon_{-p_i} -
\delta^\epsilon_{p_i})
\]
\end{definition}

Note that~$f^\epsilon_\Se$ is a smooth odd function.  We will now
prove that for sufficiently small~$\epsilon>0$ the
function~$f^\epsilon_\Se$ is \good\!.

\begin{lemma}
\label{ballestimat}
For every~$v\in US^2$ there is an~$\epsilon(v)$ with
$\epsilon(\Se)>\epsilon(v)>0$ and an open neighborhood~$U(v)$ of~$v$
in~$US^2$ (note that the base point also varies) such that for
all~$w\in U(v)$ there is a geodesic segment~$\tau$ of length~$\pi$
whose initial tangent vector makes an acute angle with~$w$ (hence it
starts at the base point of~$w$) and
\[
\int_\tau f^\epsilon_\Se<0
\]
for all~$\epsilon(v)>\epsilon>0$.
\end{lemma}

Note that the~$\epsilon(v)>0$ we find in the proof below will tend
to~$0$ as the base point of~$v$ tends to~$\Se$ (while not being
in~$\Se$).  This turns out not to be a problem since by the
compactness of~$US^2$ a finite number~$U(v_i)$ cover~$US^2$ and hence
we can take any~$\epsilon$ less than the smallest of
the~$\epsilon(v_i)$ and have that~$f^\epsilon_\Se$ is \good\!.  This
therefore proves Proposition~\ref{existgood}.

\begin{proof}[Proof of Lemma \ref{ballestimat}]
To prove the lemma, we consider two cases.  First assume that the base
point of~$v$ is not in~$\Se\cup -\Se$.  Let~$\tau$ (whose existence is
promised in Lemma \ref{thereistau}) be a geodesic segment of length
$\pi$ making an acute angle with~$v$ that misses~$-\Se$ and passes
through at least one~$p\in \Se$ that has distance greater than
$\epsilon(\Se)$ from its endpoints. Thus we can choose~$\epsilon(v)$
so small that~$\tau$ misses~$B(q,2\epsilon(v))$ for all~$q\in -\Se$.
Now for~$w$ in a small enough neighborhood~$U$ of~$v$, let~$\bar \tau$
be the geodesic segment of length~$\pi$ through the base point of~$w$
and~$p$.  For small enough~$U$,~$\bar \tau$ will still miss all
the~$B(q,\epsilon(v))$ for~$q\in -\Se$ while~$\bar \tau$ will make an
acute angle with~$w$.  Thus, for~$\epsilon(v)\geq \epsilon>0$, we
have~$f^\epsilon_\Se\leq 0$ along~$\bar \tau$ and is negative near~$p$
so we see~$\int_{\bar \tau} f^\epsilon_\Se < 0$.

In the second case the basepoint~$p_0$ of~$v$ is in~$\Se\cup -\Se$. We
will assume~$p_0\in -\Se$ since the other case is the same (by
reversing orientation of all geodesics).  Note that any~$\tau$ making
an acute angle with~$v$ which intersects~$\Se$ in its interior cannot
also intersect~$-\Se$ in its interior by property 1 of a fine set.  So
there are two (in fact three) geodesic segments~$\tau_1$ and~$\tau_2$
in the interior of~$S\ppi_{\gamma}$ that pass through a~$p_{1}$ and
a~$p_{2}$ respectively and no element of~$-\Se$ in its interior and
such that~$p_{1}$ and~$p_{2}$ have distance greater
than~$\epsilon(\Se)$ from~$p_0$ and~$-p_0$.  Now we
choose~$\epsilon(v)>0$ so small that for all~$q\in -\Se$ and~$q\not=
p_0$,~$B(q,2\epsilon(v))$ miss both~$\tau_1$ and~$\tau_2$.  Again for
$w$ in a small neighborhood~$U$ of~$v$ let ~$\bar \tau_1$ (resp~$\bar
\tau_2$) be the geodesic segment of length~$\pi$ starting at the
basepoint of~$w$ and passing through~$p_{1}$ (resp.~$p_{2})$.

If~$U$ is small enough we can assume both that~$\tau_i$ make acute
angles with~$w$ and that for~all~$q\in -\Se$ with~$q\not= p_0$,~$\bar
\tau_1$ and~$\bar \tau_2$ miss~$B(q,\epsilon(v))$.  Along~$\bar
\tau_1$ (resp.~$\bar \tau_2$) we have, for
$\epsilon(v)\geq\epsilon>0$,~$f^\epsilon_\Se\leq 0$ except on
\[
\bar \tau_1 \cap B(p_0,\epsilon)
\]
(respectively, on~$\bar \tau_2 \cap B(p_0,\epsilon)$).  Both~$\bar
\tau_1 \cap B(p_0,\epsilon)$ and~$\bar \tau_2 \cap B(p_0,\epsilon)$
cannot be diameters since that would put~$p_0$,~$p_{1}$ and~$p_{2}$ on
the same great circle (namely the one through~$p_0$ and the base point
of~$w$). So assume~$\bar \tau_1 \cap B(p_0,\epsilon)$ is not a
diameter.  Then since we know on the other hand that~$\bar \tau_1 \cap
B(p_{1},\epsilon)$ is a diameter, Lemma \ref{ballestimat} tells us
that~$\int_{\bar \tau_1} f^\epsilon_\Se <0$.
\end{proof}

\vfill\eject

\end{document}